\documentclass[11pt,a4paper]{article}
\usepackage{amsmath,amssymb}
\usepackage{amsthm}
\usepackage{longtable}
\usepackage[english]{babel}
\usepackage{array} 
\usepackage{multirow}
\usepackage{mathptmx}
\usepackage[T1]{fontenc}
\usepackage{xcolor} 
\usepackage[pagewise,mathlines]{lineno}
\nolinenumbers

\usepackage{fullpage}
\numberwithin{equation}{section}
\newtheorem{lemma}{Lemma}
\newtheorem{theorem}{Theorem}
\newtheorem{corollary}{Corollary}
\newtheorem{remark}{Remark}

\definecolor{te}{HTML}{000000}% 000000 black F62217
\newcommand{\ww}[1]{\textcolor{te}{#1}}

\def \N {\mathbb{N}}
\def \Q {\mathbb{Q}}
\def \Z {\mathbb{Z}}

\begin{document}
\title{On simultaneous approximation of the values of certain Mahler functions}
\author{Keijo V\"{a}\"{a}n\"{a}nen \and Wen Wu\footnote{Corresponding author.}}
\maketitle
\begin{abstract}
In this paper, we estimate the simultaneous approximation exponents of the values of certain Mahler functions. For this we construct Hermite-Pad\'{e} approximations of the functions under consideration, then apply the functional equations to get an infinite sequence of approximations and use the numerical approximations obtained from this sequence. 
\end{abstract}
\let\thefootnote\relax\footnotetext{The research of Wen Wu was supported by NSFC (Grant No. 11401188) and by the Academy of Finland, the Centre of Excellence in Analysis and Dynamics Research.}
\let\thefootnote\relax\footnotetext{\emph{2010 Mathematics Subject Classification:} 11J13, 11J72, 39B32. }
\let\thefootnote\relax\footnotetext{\emph{Keywords:} simultaneous approximation, Mahler function, Hermite-Pad\'{e} approximation.}

\section{Introduction}
Let $\alpha_{1},\dots, \alpha_{n}$ be real numbers. The \ww{\emph{simultaneous approximation exponent}} $\mu(\alpha_{1},\dots, \alpha_{n})$ of $\alpha_{1},\dots, \alpha_{n}$ is the supremum of the real numbers $\mu$ such that the inequality
\[\max_{1\leq i \leq n}\left|\alpha_{i}-\frac{p_{i}}{q}\right|<\frac{1}{q^{\mu}}\]
has infinitely many solutions in rational numbers $p_{i}/q$. If at least one of $\alpha_{i}$ is irrational, then $\mu\geq 1+1/n$, see e.g. W. Schmidt \cite{Schmidt1980}. In the case $n=1$, $\mu(\alpha_{1})$ is called the irrationality exponent of $\alpha_{1}$. Recently a remarkable progress has been achieved in proving that $\mu(\alpha_{1})=2$ for many classes of so-called automatic numbers and more generally the values of Mahler functions, see in particular \cite{Bugeaud2011, BHWY2015, Coons2013, GWW2014, Keijo2015, WW2014} and the references there in. In the present work our purpose is to study the simultaneous approximation exponents $\mu(\alpha_{1},\alpha_{2})$ for some numbers $\alpha_{1}$ and $\alpha_{2}$ of the above mentioned type.

Our first result considers generating functions of Stern's sequence $(a_{n})_{n\geq 0}$ and its twisted version $(b_{n})_{n\geq 0}$ defined by the recursions  
\begin{equation}
\left\{
\begin{array}{lll}
a_{0}=0, & a_{1}=1, &\\
a_{2n}=a_{n}, & a_{2n+1}=a_{n}+a_{n+1}, & (n\geq 1),
\end{array}
\right.
\end{equation}
and 
\begin{equation}
\left\{
\begin{array}{lll}
b_{0}=0, & b_{1}=1, & \\
b_{2n}=-b_{n}, & b_{2n+1}=-(b_{n}+b_{n+1}), & (n\geq 1). 
\end{array}
\right.
\end{equation}
It is proved in \cite[Theorem 2.4]{BHWY2015} that $\mu(A(1/b))=\mu(B(1/b))=2$ for all integers $b\geq 2$, where
\[A(z)=\sum_{n\geq 0}a_{n+1}z^{n} \text{ and } B(z)=\sum_{n\geq 0}b_{n+1}z^{n}.\]
These two functions satisfy the following Mahler type functional equations
\begin{equation}
A(z)=(1+z+z^{2})A(z^{2}) \text{ and } B(z)=2-(1+z+z^{2})B(z^{2}), \label{eqn:s}
\end{equation}
see also \cite{BK2013}. For $A(1/b)$ and $B(1/b)$ we have 

\begin{theorem}\label{thm:1}
For all integers $b\geq 2$, \[\mu\left(A\left(\frac{1}{b}\right),B\left(\frac{1}{b}\right)\right)\leq \frac{8}{5}=1.6.\] Moreover, if $a/b\in \Q$ with $\log |a|=\lambda\log b$, where $b\geq 2$, $0\leq \lambda < 50/77$, then \[\mu\left(A\left(\frac{a}{b}\right),B\left(\frac{a}{b}\right)\right)\leq \frac{80(1-\lambda)}{(50-77\lambda)}.\] 
\end{theorem}

Now we turn to the following two power series
\[T(z)=\prod_{j=0}^{\infty}(1-z^{2^{j}}) \text{ and } M(z)=\sum_{j=0}^{\infty}\frac{(-1)^{j}z^{2^{j}}}{\prod_{i=0}^{j-1}(1-z^{2^{j}})},\]
in particular $T(z)$ is the generating function of the Thue-Morse sequence on $\{-1,1\}$. 
These series \ww{are solutions of} Mahler type functional equations 
\begin{equation}
T(z)=(1-z)T(z^{2}) \text{ and } M(z^{2})=(z-1)(M(z)-z), \label{eqn:tm}
\end{equation}
and the coefficients of $T(z)=\sum_{j=0}^{\infty}t_{j}z^{j}$,  $M(z)=\sum_{j=0}^{\infty}m_{j}z^{j}$ satisfy, for \ww{all} $n\geq 1$,\\
\begin{tabular}{ccc}
\begin{minipage}{.4\textwidth}
\[\left\{\begin{aligned}
& t_{0}=1, \\
& t_{2n}=t_{n},\\
& t_{2n+1}=-t_{n}, 
\end{aligned} 
\right.\]
\end{minipage}
\begin{minipage}{.1\textwidth}
and
\end{minipage}
\begin{minipage}{.45\textwidth}
\[\left\{\begin{aligned}
& m_{0}=0, \\
& m_{1}=-m_{2}=1,\\
& m_{2n+1}=m_{2n}, \\
& m_{2n+2}=m_{2n+1}-m_{n+1}.
\end{aligned}
\right.\]
\end{minipage}
\end{tabular}\\
\ww{The numbers} $T(\alpha)$ and $M(\alpha)$ are algebraically independent over $\mathbb{Q}$ for any non-zero algebraic number $\alpha$ with $|\alpha|<1$, \ww{see \cite[Theorem 4]{Keijo2015}}. For all integers $b\geq 2$, Bugeaud \cite{Bugeaud2011} proved that $\mu(T(b^{-1}))$ equals $2$ and V\"{a}\"{a}n\"{a}nen \cite{Keijo2015} proved that $\mu(M(b^{-1}))$ also equals $2$. Thus we have $$\frac{3}{2}\leq \mu\left(T\left(\frac{1}{b}\right),M\left(\frac{1}{b}\right)\right)\leq \min\left\{\mu\left(T\left(\frac{1}{b}\right)\right),\mu\left(M\left(\frac{1}{b}\right)\right)\right\}= 2.$$ 
The following result improves the above upper bound.
\begin{theorem}\label{thm:TM}
For all integers $b\geq 2$, \[\mu\left(T\left(\frac{1}{b}\right),M\left(\frac{1}{b}\right)\right)\leq \frac{32}{17}= 1.882\dots\]
Moreover, if $a/b\in \Q$ with $\log |a|=\lambda\log b$, $b\geq 2$, $0\leq \lambda <1/2 $, then \[\mu\left(T\left(\frac{a}{b}\right),M\left(\frac{a}{b}\right)\right)\leq \frac{32(1-\lambda)}{17-26\lambda}.\]
\end{theorem}

Our next result studies a special type of Lambert series 
\[G_{d}(z)=\sum_{k=0}^{\infty}\frac{z^{d^{k}}}{1-z^{d^{k}}}\]
and a related function
\[F_{d}(z)=\sum_{k=0}^{\infty}\frac{z^{d^{k}}}{1+z^{d^{k}}}\]
with $d=3$. These series satisfy the Mahler type functional equations
\begin{equation}
\left\{
\begin{array}{rcl}
(z-1)G_{d}(z)+(1-z)G_{d}(z^{d})+z & = & 0, \\
-(1+z)F_{d}(z)+(1+z)F_{d}(z^{d})+z & = & 0. 
\end{array}\right. \label{eqn:G}
\end{equation}

Recently Coons \cite{Coons2013} proved that $\mu(G_{2}(1/b))=\mu(F_{2}(1/b))=2$. The numbers $1$, $G_{2}(1/b)$ and $F_{2}(1/b)$ are linearly dependent, namely \[G_{2}(\alpha)+F_{2}(\alpha)=\frac{2\alpha}{1-\alpha}\] for all $|\alpha|<1$. Thus we get $\mu(G_{2}(1/b),F_{2}(1/b))=2$. On the other hand, for $d\geq 3$ and algebraic $\alpha$, $0<|\alpha|<1$, two numbers $G_{d}(\alpha)$ and $F_{d}(\alpha)$ are known to be algebraically independent, see e.g. \cite{BK2015}. Here we are interested in the simultaneous approximation of $G_{3}(a/b)$ and $F_{3}(a/b)$, but we note that our approximation lemma \ww{in section 2} could also be used to estimate irrationality exponents of these numbers.

\begin{theorem}\label{thm:2}
For all integers $b\geq 2$, \[\mu\left(G_{3}\left(\frac{1}{b}\right),F_{3}\left(\frac{1}{b}\right)\right)\leq \frac{36}{19}=1.894\dots\] Moreover, if $a/b\in\Q$ with $\log |a|=\lambda\log b$, $b\geq 2$, $0\leq \lambda < 19/29$, then \[\mu\left(G_{3}\left(\frac{a}{b}\right),F_{3}\left(\frac{a}{b}\right)\right)\leq \frac{36(1-\lambda)}{(19-29\lambda)}.\]
\end{theorem}

The following result considers the power series 
\[S(z)=1+z+z^{3}+z^{4}+z^{5}+z^{11}+z^{12}+z^{13}+z^{16}+z^{17}+z^{19}+\cdots\]
introduced by Dilcher and Stolarsky \cite{DS2009} and satisfying the Mahler type functional equation
\begin{equation}
S(z^{16})=-zS(z)+(1+z+z^{2})S(z^{4}) \label{eqn:g}
\end{equation} 
of degree $2$. This series is connected to Stern polynomials and we see immediately by (\ref{eqn:g}) that the coefficients $s_{k}$ in $S(z)=\sum_{k=0}^{\infty}s_{k}z^{k}$ satisfy $s_{0}=1$ and, for all $k\geq 0$,
\[s_{4k}=s_{4k+1}=s_{k},\ s_{4k+2}=0,\ s_{4k+3}=
\begin{cases}
0, & k\equiv 3 \pmod 4,\\
s_{k+1}, & \text{otherwise.}
\end{cases}
\]
In particular, $s_{k}\in\{0,1\}$ and, moreover, the indexes $k$ with $s_{k}=1$ form a so-called self-generating set, see \cite{DS2009}. In \cite{Adam2010} Adamczewski proved that the numbers $S(\alpha)$ and $S(\alpha^{4})$ are algebraically independent, if $\alpha$, $0<|\alpha|<1$, is algebraic. Further, by using the gap properties of the series $S(z)$ an upper bound $\mu(S(a/b))<(5-2\lambda)/(1-2\lambda)$, $b\geq 2$, $0\leq \lambda <1/2$, is proved in \cite{BK2014}. Now we prove
\begin{theorem}\label{thm:self}
For all integers $b\geq 2$, \[\mu\left(S\left(\frac{1}{b}\right), S\left(\frac{1}{b^{4}}\right)\right)\leq \frac{516}{253}= 2.039\dots\]
Moreover, if $a/b\in \Q$ with $\log |a|=\lambda\log b$, $b\geq 2$, $0\leq \lambda < 178/291$, then  \[\mu\left(S\left(\frac{a}{b}\right), S\left(\frac{a^{4}}{b^{4}}\right)\right)\leq {\frac {516\,(1-\lambda)}{253-381\,\lambda}}.\]
\end{theorem} 

\ww{
The above results on simultaneous approximation and Khintchine's transference theorem (see e.g. \cite{Schmidt1980}, p. 99) give some information on linear independence exponents $\mu_{L}(\alpha_{1},\alpha_{2})$ defined as the supremum of real numbers $\mu$ such that the inequality
\[
|h_{0}+h_{1}\alpha_{1}+h_{2}\alpha_{2}| < h^{-\mu}
\]
with $h = \max\{\left|h_i\right|\}$ has infinitely many solution $(h_{0},h_{1},h_{2}) \in \mathbb{Z}^3\setminus{\{\underline{0}\}}$. Namely, if $\mu(\alpha_{1},\alpha_{2})\leq U<2$, then 
\[\mu_{L}(\alpha_{1},\alpha_{2})\leq \frac{2}{2-\mu(\alpha_{1},\alpha_{2})}-2\leq \frac{2}{2-U}-2.\]
 For example, our Theorems \ref{thm:1}-\ref{thm:2} imply the following 
\begin{corollary}
For all integers $b\geq 2$ we have
\[
\mu_{L}\left(A\left(\frac{1}{b}\right),B\left(\frac{1}{b}\right)\right)\leq 3, \ \mu_{L}\left(T\left(\frac{1}{b}\right),M\left(\frac{1}{b}\right)\right)\leq 15, \ \mu_{L}\left(G_{3}\left(\frac{1}{b}\right),F_{3}\left(\frac{1}{b}\right)\right)\leq 17.
\]
\end{corollary}
However, in a very recent work \cite{VM2016} we obtained better results $\mu_{L}\left(A\left(\frac{1}{b}\right),B\left(\frac{1}{b}\right)\right)\leq \frac{26}{9} = 2.888\ldots$ and $\mu_{L}\left(G_{3}\left(\frac{1}{b}\right),F_{3}\left(\frac{1}{b}\right)\right)\leq \frac{129}{37} = 3,486\ldots$ by using another method based on the ideas of Siegel's method.
}

\medskip
This paper is organized as follows. In Section 2, basic information about Hermite-Pad\'{e} approximations and an important approximation lemma are given \ww{in a general form for arbitrary number of functions. This lemma gives good simultaneous approximation exponents only if one knows well the asymptotic bounds for linear forms obtained by using Hermite-Pad\'{e} approximations, and this is generally hard. In the case of two functions studied in our theorems we compute explicitly some approximations and the order (at $z = 0$) of the remainder terms. This computational information is given in \cite[Appendix]{VM2015}. As explained at the end of subsection 2.1 we can then produce an infinite sequence of approximations, where the coefficient polynomials and remainders are well controlled and give all we need for the application of the approximation lemma. In Section 3 this application leads to the proof of Theorems 1 to 3 considering Mahler functions of degree one.} The proof of Theorem 4, which is more complicated, is presented in Section 4.

\section{Preliminaries and approximation lemma}
\subsection{Hermite-Pad\'{e} approximation}\label{sec:pade}
In this paragraph, we introduce some basic information about the Hermite-Pad\'{e} approximation to be needed in the following. For general theory, see for example \cite{baker1996,Nikishin1991}. 

Let \ww{$n\geq 1 $ be an integer and $f_{1}(z),\, f_{2}(z), \cdots,\, f_{n}(z)\in\mathbb{Q}[[z]]$} be formal power series. For given $\mathbf{d}:=(d_{1},\, d_{2},\,\ww{\cdots,\, d_{n})\in \mathbb{Z}_{\geq 0}^{n}}$, let  $P_{1}(z)$, $P_{2}(z)$, $\ww{\cdots,\, P_{n}(z)}\in\mathbb{Z}[[z]]$ be non-trivial polynomials such that $\deg P_{i}\leq d_{i}\, (i=1,2,\ww{\cdots, n})$ and  
\[P_{1}(z)f_{1}(z)+P_{2}(z)f_{2}(z)+ \ww{\cdots + P_{n}(z)f_{n}(z)}=R_{\mathbf{d}}(z),\]  
where the order of zero at $z=0$ of the \emph{remainder term} $R_{\mathbf{d}}(z)$, say $\mathrm{ord}R_{\mathbf{d}}(z)$, is at least $s+1$ where $s:=\ww{n-2}+\sum_{i=1}^{\ww{n}}d_{i}$. Such polynomials exist since, using the notations
\[f_{j}(z)=\sum_{i=0}^{\infty}f_{i}^{(j)}z^{i} \text{ and } P_{j}(z)=\sum_{i=0}^{d_{j}}p_{i}^{(j)}z^{i},\]
the following system of linear equations
\begin{equation}
\begin{bmatrix}
f_{0}^{(1)} & 0 & \cdots & 0 & & f_{0}^{(\ww{n})} & 0 & \cdots & 0 \\
f_{1}^{(1)} & f_{0}^{(1)} & \cdots & 0 & & f_{1}^{(\ww{n})} & f_{0}^{(\ww{n})} & \cdots & 0 \\
\vdots & \vdots & \ddots & \vdots & & \vdots & \vdots & \ddots & \vdots\\
f_{d_{1}}^{(1)} & f_{d_{1}-1}^{(1)} & \cdots & f_{0}^{(1)} & \cdots & f_{d_{\ww{n}}}^{(\ww{n})} & f_{d_{\ww{n}}-1}^{(\ww{n})} & \cdots & f_{0}^{(\ww{n})}\\
\vdots & \vdots &  & \vdots &  & \vdots & \vdots &   & \vdots \\
f_{s}^{(1)} & f_{s-1}^{(1)} & \cdots & f_{s-d_{1}}^{(1)} & & f_{s}^{(\ww{n})} & f_{s-1}^{(\ww{n})} & \cdots & f_{s-d_{\ww{n}}}^{(\ww{n})}\\
\end{bmatrix}
\cdot 
\begin{bmatrix}
p_{0}^{(1)} \\ \vdots \\ p_{d_{1}}^{(1)} \\ \vdots \\ p_{0}^{(\ww{n})} \\ \vdots \\ p_{d_{\ww{n}}}^{(\ww{n})} \\
\end{bmatrix}
= \mathbf{0}_{(s+1)\times 1} \label{eqn:pade}
\end{equation} 
has a non-trivial solution $p_{i}^{(j)}$. 
Denote the $(s+1)\times (s+2)$ coefficient matrix in (\ref{eqn:pade}) by $\Delta_{\mathbf{d}}$. Let \[X:=(p_{0}^{(1)} , \cdots, p_{d_{1}}^{(1)}, \cdots,  p_{0}^{(\ww{n})}, \cdots, p_{d_{\ww{n}}}^{(\ww{n})})^{T}\] be a non-trivial solution of (\ref{eqn:pade}) and \[\delta_{j}:=(f_{s+j}^{(1)},\, f_{s+j-1}^{(1)},\, \cdots,\, f_{s+j-d_{1}}^{(1)},\, \cdots,\, f_{s+j}^{(\ww{n})},\,  f_{s+j-1}^{(\ww{n})},\, \cdots,\, f_{s+j-d_{\ww{n}}}^{(\ww{n})}),\, j\geq -s,\]
where $f_{i}^{(j)}=0$ if $i<0$. 
Then \[R_{\mathbf{d}}(z)=\sum_{j=1}^{\infty}(\delta_{j} X)\cdot z^{s+j}.\] The exact order of $R_{\mathbf{d}}(z)$ is $\min \{s+j:\delta_{j} X \neq 0\}$. Essential for the existence of $(d_{1},d_{2},\ww{\ldots,d_{n}})$ Hermite-Pad\'{e} approximation of exact order $s+1$ is \[\delta_{1} X\neq 0.\] This condition is certainly satisfied if the determinant 
\begin{equation}
\begin{vmatrix}
\Delta_{\mathbf{d}}\\ \delta_{1}
\end{vmatrix}\neq 0. \label{eqn:det}
\end{equation}

\ww{To get explicit simultaneous approximation exponents we need to control well the orders of remainder terms $R_{\mathbf{d}}(z)$. This is the main reason why we consider in the present paper only two Mahler functions, say $f(z)$ and $g(z)$. So we choose above $n = 3, f_1(z) = f(z), f_2(z) = g(z)$ and $f_3(z) = 1$, take a positive integer $k$ and construct $(d_1,d_2,d_3) := (d_1(k),d_2(k),d_3(k))$ approximation polynomials $A_{k}(z)$, $B_{k}(z)$, $C_{k}(z)\in \Z [z]$, not all zero and of degree $\leq d_{1}$, $d_{2}$, $d_{3}$, respectively, such that
\[A_{k}(z)f(z)+B_{k}(z)g(z)+C_{k}(z)=R_{k}(z),\]
where $\mathfrak{o}(k):=\mathrm{ord}R_{k}(z) \geq d_{1}+d_{2}+d_{3}+2$ is computed explicitly, see \cite[Appendix]{VM2015}. This will be done for $k=k_{1}, \dots, k_{t}$. If we replace above $z$ by $z^d$ ($d = 2$ in Theorems 1 and 2, 3 in Theorem 3 and 4 in Theorem 4) and use functional equations, we get a new approximation form. Repeating this we obtain, for each $k$, an infinite sequence of approximations 
\[A_{k,m}(z)f(z)+B_{k,m}(z)g(z)+C_{k,m}(z)=R_{k,m}(z),\ m=0,1,\dots,\]
where the order of $R_{k,m}(z)$ is known. These sequences at $z = a/b$ give then numerical approximation sequences for $f(a/b)$ and $g(a/b)$ used in our proofs.} 

\subsection{Approximation lemma} 
In the following we shall give our main tool for the proofs, the approximation lemma tailored suitable for the above situation. For this lemma we arrange the pairs $(k,m)$, where $m\in\N, m\geq m_{0}\in \N$, and $k\in\{k_{1},\dots, k_{t}\}$, $k_{j}\in \N$, $k_{1}<k_{2}<\cdots < k_{t}$, as follows:
\begin{equation}\label{eqn:a}
(k_{1},m_{0}), \dots, (k_{t},m_{0}), (k_{1},m_{0}+1), \dots, (k_{t}, m_{0}+1), (k_{1},m_{0}+2), \dots
\end{equation}
Note that in the case $u=1$ our lemma gives an upper bound for the irrationality exponent of $\gamma_{1}$, which is a kind of refinement of a result of Adamczewski and Rivoal \cite[Lemma 4.1]{AR2009}. \ww{We shall present the lemma in a general form although only the case $u = 2$ is needed in our proofs below.}
 
\begin{lemma}\label{lem:1}
Let $\gamma_{1}, \dots, \gamma_{u}$ be real numbers. Assume that for each pair $(k,m)$ in $(\ref{eqn:a})$ there exist a linear form
\[r(k,m)=h_{0}+\sum_{i=1}^{u}h_{i}\gamma_{i}, \quad h_{i}=h_{i}(k,m)\in\Z,\] 
and a positive integer $Q_{k,m}$ such that %
\begin{eqnarray}
& c_{1}(k)Q_{k,m} \leq \max\limits_{1\leq i\leq u}|h_{i}| \leq c_{2}(k)Q_{k,m}, &\label{eqn:l1}\\
& Q_{k_{j},m} < Q_{k_{j+1},m}\leq C_{1}(\underline{k})Q_{k_{j},m}^{\theta(j)},& ~j=1,\ \dots,\ t-1, \label{eqn:l2}\\
& Q_{k_{t},m} <Q_{k_{1},m+1}\leq C_{1}(\underline{k})Q_{k_{t},m}^{\theta(t)},& \label{eqn:l3}\\
& c_{3}(k)Q_{k,m}^{-\alpha(k)} \leq \left|r(k,m)\right|\leq c_{4}(k)Q_{k,m}^{-\beta(k)}, &\label{eqn:l4}%
\end{eqnarray}
where $c_{i}(k)$ and $C_{i}(\underline{k})$ are positive constants depending on $k$ and $\underline{k}=\{k_{1},\dots, k_{t}\}$, respectively, and $\theta(j)\geq 1$, $\alpha(k_{j})\geq \beta(k_{j})>0$ are constants (all independent of $m$). Then there exist positive constants $Q_{0}=Q_{0}(\underline{k},m_{0})$ and $C=C(\underline{k},m_{0})$ such that, for all $\frac{p_{i}}{q}\in\Q, q\geq Q_{0}$, 
\[\max_{1\leq i\leq u}\left|\gamma_{i}-\frac{p_{i}}{q}\right|>Cq^{-\mu}\]
where $\mu=\max\limits_{1\leq j\leq t}\theta(j)\dfrac{\alpha(k_{j+1})+1}{\beta(k_{j})}$,  $\alpha(k_{t+1}):=\alpha(k_{1})$.
\end{lemma}

\begin{proof}
For the proof we denote
\[\Delta =\Delta(k,m):=h_{0}+\sum_{i=1}^{u}h_{i}\frac{p_{i}}{q}.\]
By defining $\varepsilon_{i}:=\gamma_{i}-\frac{p_{i}}{q} ~(i=1, \dots, u)$ we have
\[\Delta=r(k,m)-\sum_{i=1}^{u}h_{i}\varepsilon_{i}.\]
Assume that, for some $(k,m)$ in $(\ref{eqn:a})$, 
\begin{equation}
\max_{1\leq i\leq u}|\varepsilon_{i}|<\left(\frac{\min_{j}c_{3}(k_{j})}{2u\max_{j}c_{2}(k_{j})}\right)Q_{k,m}^{-\alpha(k)-1}=:C_{2}(\underline{k})Q_{k,m}^{-\alpha(k)-1}. \label{eqn:lem1:1}
\end{equation}
This implies, by (\ref{eqn:l1}) and (\ref{eqn:l4}) 
\[\left|\sum_{i=1}^{u}h_{i}\varepsilon_{i} \right|\leq \frac{c_{3}(k)}{2}Q_{k,m}^{-\alpha(k)}\leq \frac{|r(k,m)|}{2},\]
and therefore \[0<|\Delta|\leq \frac{3|r(k,m)|}{2}\leq \frac{3c_{4}(k)}{2}Q_{k,m}^{-\beta(k)}.\]
Moreover, since $h_{i}\in\Z$, we have $|\Delta|\geq 1/q$. Thus
\begin{equation}
\log q\geq \beta(k)\log Q_{k,m}-\log \frac{3c_{4}(k)}{2}. \label{eqn:lem1:2}
\end{equation}

Let us define $Q_{0}=Q_{0}(\underline{k},m_{0})$ in such a way that
\[\log Q_{0}>\left(\max_{1\leq j\leq t}\beta(k_{j})\right)\log Q_{k_{1},m_{0}}-\left(\min_{1\leq j\leq t}\log\frac{3c_{4}(k_{j})}{2}\right).\]
We now assume that $q\geq Q_{0}$ and fix the pair $(k_{j},m)$ from $(\ref{eqn:a})$ such that it is the first one satisfying
\begin{equation}
\log q<\beta(k_{j})\log Q_{k_{j},m}-\log \frac{3c_{4}(k_{j})}{2}, \label{eqn:lem1:3}
\end{equation}
by (\ref{eqn:l2}) and (\ref{eqn:l3}) such pair exists. From the above choice of $Q_{0}$ it follows that $(k_{j},m)\neq (k_{1},m_{0})$. Then $(\ref{eqn:lem1:1})$ implying $(\ref{eqn:lem1:2})$ cannot hold for the pair $(k_{j},m)$ and therefore
\begin{equation}
\max_{1\leq i\leq u}|\varepsilon_{i}|\geq C_{2}(\underline{k})Q_{k_{j},m}^{-\alpha(k_{j})-1}. \label{eqn:lem1:4}
\end{equation}

From the definition of $(k_{j},m)$ it follows that the pair just before it does not satisfy $(\ref{eqn:lem1:3})$. For $j>1$ this pair is $(k_{j-1},m)$, and for $j=1$ it is $(k_{t},m-1)$. Thus in the first case
\begin{eqnarray*}
\log q & \geq & \beta(k_{j-1})\log Q_{k_{j-1},m}-\log\frac{3c_{4}(k_{j-1})}{2}\\
 & \geq & \beta(k_{j-1})\log Q_{k_{j-1},m}-\log C_{3}(\underline{k}),
\end{eqnarray*}
where $C_{3}(\underline{k})=\max\limits_{1\leq j\leq t}\frac{3}{2}c_{4}(k_{j})$. By (\ref{eqn:l2}), 
\[\log Q_{k_{j},m}\leq \log C_{1}(\underline{k})+\theta(j-1)\log Q_{k_{j-1},m},\]
which then implies
\begin{eqnarray*}
\log q & \geq & \frac{\beta(k_{j-1})}{\theta(j-1)}\log Q_{k_{j},m}-\log \left(C_{1}(\underline{k})^{\frac{\beta(k_{j-1})}{\theta(j-1)}}C_{3}(\underline{k})\right)\\
& \geq &  \frac{\beta(k_{j-1})}{\theta(j-1)}\log Q_{k_{j},m}-\log C_{4}(\underline{k}).
\end{eqnarray*}
Thus \[q \geq C_{4}(\underline{k})^{-1}Q_{k_{j},m}^{\frac{\beta(k_{j-1})}{\theta(j-1)}}.\]
By $(\ref{eqn:lem1:4})$ we now obtain
\[\max_{1\leq i\leq u}|\varepsilon_{i}|\geq C_{5}(\underline{k})q^{-\mu(j)},\]
where $\mu(j)=\theta(j-1)\frac{\alpha(k_{j})+1}{\beta(k_{j-1})}$. Similarly, by using (\ref{eqn:l3}), we get in the second case
\[\max_{1\leq i\leq u}|\varepsilon_{i}|\geq C_{6}(\underline{k})q^{-\mu(1)},\] 
where $\mu(1)=\theta(t)\frac{\alpha(k_{1})+1}{\beta(k_{t})}$. These two estimates prove the truth of our lemma.
\end{proof}

\section{Proof of Theorems \ref{thm:1}, \ref{thm:TM} and \ref{thm:2}}
In proving Theorems \ref{thm:1}, \ref{thm:TM} and \ref{thm:2}, we shall consider simultaneous approximations of Mahler functions $f(z)=\sum_{j=0}^{\infty}f_{j}z^{j}$ and $g(z)=\sum_{j=0}^{\infty}g_{j}z^{j}$ converging in the open unit disc and satisfying functional equations of the type
\begin{eqnarray}
\phi_{1}(z)f(z)+ \phi_{2}(z)f(z^{d})+ \phi_{3}(z)  =  0,\label{eqn:f4}\\
\psi_{1}(z)g(z)+ \psi_{2}(z)g(z^{d})+ \psi_{3}(z)  =  0,\label{eqn:g4}
\end{eqnarray}
where $\phi_{i},\, \psi_{i}\in \mathbb{Z}[z]$ and $d\geq 2$ is a fixed integer. We assume that  $|\phi_{i}(0)|=|\psi_{i}(0)|=1\, (i=1,2)$.  Denote the zeros of $\phi_{i},\, \psi_{i}\, (i=1,2)$ by 
\[\Xi=\{z\in\mathbb{R},|z|<1~|~  \phi_{i}(z)=0 \text{ or } \psi_{i}(z)=0,\, i=1,2\}.\]

Theorems \ref{thm:1}, \ref{thm:TM} and \ref{thm:2} are obtained from the following general Theorem \ref{lem:2} on simultaneous approximations of the values of $f(z)$ and $g(z)$ at rational points. For it we need some notations. Let $\phi(z)=\mathrm{l.c.m}(\phi_{2}(z),\psi_{2}(z))$ in $\mathbb{Z}[z]$ and define the polynomials $\hat{\phi}_{2}$, $\hat{\psi}_{2}\in\mathbb{Z}[z]$ by  $\phi(z)=\phi_{2}(z)\hat{\psi}_{2}(z)=\psi_{2}(z)\hat{\phi}_{2}(z)$. Moreover, let $v$ denote the maximal degree of the polynomials $\phi_{1}\hat{\psi}_{2}$, $\hat{\phi}_{2}\psi_{1}$, $\hat{\phi}_{2}\psi_{3}$, $\phi_{3}\hat{\psi}_{2}$ and $\phi(z)$.
Let $A_{k}(z)$, $B_{k}(z)$ and $C_{k}(z)$ be the $(d_{1},\, d_{2},\, d_{3}):=(d_{1}(k),\, d_{2}(k),\, d_{3}(k))$ simultaneous approximation polynomials of $f(z)$, $g(z)$ and $1$, and denote the exact order of the remainder term by $\ww{\mathfrak{o}}(k)$. Let $\bar{d}(k):=\max\{d_{1}(k),\, d_{2}(k),\, d_{3}(k)\}$ and assume that $\bar{d}(k)$ is strictly increasing.

\begin{theorem}\label{lem:2}
Let $f(z)$ and $g(z)$ be the functions given above.
Suppose that, for $k\in \{k_{1}, \dots, k_{t}\}$ with $k_{1}<\cdots <k_{t}$, the $(d_{1},\, d_{2},\, d_{3})$ simultaneous approximation polynomials satisfy
\begin{enumerate}
\item[(i)] at least one of $A_{k}(0)$ and $B_{k}(0)$ is not zero.
\end{enumerate}
There exist non-negative integers $e_{1}$ and $e_{2}$, depending on (\ref{eqn:f4}) and (\ref{eqn:g4}) such that if
\begin{enumerate}
\item[(ii)] $d\cdot \bar{d}(k_{1})+v - \bar{d}(k_{t})> (d-1)e_{1}+e_{2}$ 
\end{enumerate}
and $a/b\in \Q$ with $(a/b)^{l}\notin \Xi\, (l\in \mathbb{N})$ and $\log |a|=\lambda\log b$, $b\geq 2$, \[0\leq \lambda < \min\limits_{1\leq j\leq t}\left\{\frac{\ww{\mathfrak{o}}(k_{j})-(\bar{d}(k_{j})-e_{1})-\frac{v-e_{2}}{d-1}}{\ww{\mathfrak{o}}(k_{j})}\right\},\]
then \[\mu\left(f\left(\frac{a}{b}\right),g\left(\frac{a}{b}\right)\right)\leq \max_{1\leq j \leq t} \left\{ \frac{(1-\lambda)\ww{\mathfrak{o}}(k_{j+1})}{(1-\lambda)\ww{\mathfrak{o}}(k_{j})-(\bar{d}(k_{j})-e_{1})-\frac{v-e_{2}}{d-1}}\right\},\]
where $\ww{\mathfrak{o}}(k_{t+1}):=\ww{\mathfrak{o}}(k_{1})d$.
\end{theorem}

\begin{proof}
We start from the type $(d_{1},\, d_{2},\, d_{3}):=(d_{1}(k),\, d_{2}(k),\, d_{3}(k))$ approximation polynomials $A_{k}(z)$, $B_{k}(z)$, $C_{k}(z)\in \Z [z]$ satisfying
\[A_{k}(z)f(z)+B_{k}(z)g(z)+C_{k}(z)=R_{k}(z).\]
Substituting here $z^{d}$ for $z$ and applying (\ref{eqn:f4}) and (\ref{eqn:g4}), we obtain
\begin{eqnarray*}
 \hat{\psi}_{2}(z)\phi_{1}(z)A_{k}(z^{d})f(z) +  \hat{\phi}_{2}(z)\psi_{1}(z)B_{k}(z^{d})g(z) &&\\
+  \hat{\psi}_{2}(z)\phi_{3}(z)A_{k}(z^{d})+\hat{\phi}_{2}(z)\psi_{3}(z)B_{k}(z^{d})
-\phi(z)C_{k}(z^{d}) & = & -\phi(z)R_{k}(z^{d}).
\end{eqnarray*}
Repeating this procedure $m$ times, we have
\begin{equation}
A_{k,m}(z)f(z)+B_{k,m}(z)g(z)+C_{k,m}(z)=R_{k,m}(z),\ m=0,1,\dots,  \label{eqn:linear}
\end{equation}
where $A_{k,0}(z)=A_{k}(z)$, $B_{k,0}(z)=B_{k}(z)$, $C_{k,0}(z)=C_{k}(z)$, $R_{k,0}(z)=R_{k}(z)$ and, for $m=1,2,\dots$,
\begin{eqnarray*}
A_{k,m}(z) & = & \hat{\psi}_{2}(z)\phi_{1}(z)A_{k,m-1}(z^{d}),\\
B_{k,m}(z) & = & \hat{\phi}_{2}(z)\psi_{1}(z)B_{k,m-1}(z^{d}),\\
C_{k,m}(z) & = & \hat{\psi}_{2}(z)\phi_{3}(z)A_{k,m-1}(z^{d})+\hat{\phi}_{2}(z)\psi_{3}(z)B_{k,m-1}(z^{d})\\
& & -\phi(z)C_{k,m-1}(z^{d}),\\
R_{k,m}(z) & =& -\phi(z)R_{k,m-1}(z^{d}).
\end{eqnarray*}
Notice that
\begin{equation}
\deg A_{k,m}(z),\, \deg B_{k,m}(z),\, \deg C_{k,m}(z) \leq \left(\bar{e}(k)+\frac{\tau}{d-1}\right)\cdot d^{m}-\frac{\tau}{d-1} \label{eq:x}
\end{equation}
with $\bar{e}(k)=\bar{d}(k)-e_{1}$, $\tau=v-e_{2}$, where $e_{1}$ and $e_{2}$ are non-negative integers (if $e_{1}=e_{2}=0$, then (\ref{eq:x}) certainly holds). Further, 
\[\mathrm{ord}R_{k,m}(z)=\ww{\mathfrak{o}}(k)d^{m}.\]

Using (\ref{eqn:linear}), we construct linear forms 
\[a_{k,m}f\left(\frac{a}{b}\right)+b_{k,m}g\left(\frac{a}{b}\right)+c_{k,m}=r_{k,m},\quad m=0,1,\dots\]
where 
\begin{eqnarray}
& a_{k,m}=Q_{k,m}A_{k,m}\left(\frac{a}{b}\right),\quad b_{k,m}=Q_{k,m}B_{k,m}\left(\frac{a}{b}\right), \label{lem2:ab}\\
& c_{k,m}=Q_{k,m}C_{k,m}\left(\frac{a}{b}\right),\quad r_{k,m}=Q_{k,m}R_{k,m}\left(\frac{a}{b}\right) \label{lem2:cr}
\end{eqnarray}
with $Q_{k,m}=b^{\left(\bar{e}(k)+\frac{\tau}{d-1}\right)\cdot d^{m}-\frac{\tau}{d-1}}$. Here all $a_{k,m}$, $b_{k,m}$ and $c_{k,m}$ are integers.

Suppose $\underline{k}:=\{k_{1}, \dots, k_{t}\}$ and $k_{1}<k_{2}<\cdots < k_{t}$. By our assumption, for all $k\in\underline{k}$, we have type $(d_{1},d_{2},d_{3})$ simultaneous approximation polynomials such that at least one of $A_{k}(0)$ and $B_{k}(0)$ is not zero.

Now  \[A_{k,m}\left(\frac{a}{b}\right) = A_{k}\left(\left(\frac{a}{b}\right)^{d^{m}}\right)\prod_{j=0}^{m-1}\hat{\psi}_{2}\left(\left(\frac{a}{b}\right)^{d^{j}}\right)\phi_{1}\left(\left(\frac{a}{b}\right)^{d^{j}}\right)\] 
implying, by our assumptions on $\phi_{i}$ and $\psi_{i}$, that 
\[c_{5}(k)\leq \left|A_{k,m}\left(\frac{a}{b}\right)\right|\leq c_{6}(k)\]
for all $m\geq m_{0}=m_{0}(k)$, if $A_{k}(0)\neq 0$. $B_{k,m}\left(a/b\right)$ can be estimated similarly. Thus the condition (\ref{eqn:l1}) holds for all $m\geq m_{0}$. For a given $\delta>0$ there exists $m_{1}=m_{1}(\delta)>0$ such that the conditions (\ref{eqn:l2}) and (\ref{eqn:l3}) are also satisfied for all $m\geq m_{1}$ if we choose 
\[\theta(j)=\frac{\bar{e}(k_{j+1})+\frac{\tau}{d-1}}{\bar{e}(k_{j})+\frac{\tau}{d-1}}+\delta,\, j=1,\dots, t-1, \text{ and } \theta(t)=d\cdot \frac{\bar{e}(k_{1})+\frac{\tau}{d-1}}{\bar{e}(k_{t})+\frac{\tau}{d-1}}+\delta.\]
\ww{Note that $\theta(t) > 1$ by the assumption (\emph{ii}).} Moreover 
\[R_{k,m}\left(\frac{a}{b}\right)=R_{k}\left(\left(\frac{a}{b}\right)^{d^{m}}\right)\prod_{j=0}^{m-1}\phi\left(\left(\frac{a}{b}\right)^{d^{j}}\right).\]
Since $f(z)$ and $g(z)$ converge in the open unit disc we may choose $m_{0}$ above in such a way that
\[c_{7}(k)\left(\frac{|a|}{b}\right)^{\ww{\mathfrak{o}}(k)d^{m}}\leq \left|R_{k}\left(\left(\frac{a}{b}\right)^{d^{m}}\right)\right|\leq c_{8}(k)\left(\frac{|a|}{b}\right)^{\ww{\mathfrak{o}}(k)d^{m}}\]
for all $m\geq m_{0}$. Therefore (\ref{eqn:l4}) also holds for all $m\geq m_{2}=m_{2}(\delta)$ with
\[\alpha(k)-\delta=\beta(k)=\frac{(1-\lambda)\ww{\mathfrak{o}}(k)}{\bar{e}(k)+\frac{\tau}{d-1}}-1.\]
For $j=1,\dots, t-1$,
\[
\theta(j)\frac{\alpha(k_{j+1})+1}{\beta(k_{j})}  \leq  \frac{(1+\delta)\left[(1-\lambda)\ww{\mathfrak{o}}(k_{j+1})+\delta(\bar{e}(k_{j+1})+\frac{\tau}{d-1})\right]}{(1-\lambda)\ww{\mathfrak{o}}(k_{j})-\bar{e}(k_{j})-\frac{\tau}{d-1}},
\] 
and 
\[\theta(t)\frac{\alpha(k_{1})+1}{\beta(k_{t})}  \leq \frac{(1+\delta)d\left[(1-\lambda)\ww{\mathfrak{o}}(k_{1})+\delta(\bar{e}(k_{1})+\frac{\tau}{d-1})\right]}{(1-\lambda)\ww{\mathfrak{o}}(k_{t})-\bar{e}(k_{t})-\frac{\tau}{d-1}}.\]
Applying Lemma \ref{lem:1}, we are done, since $\delta>0$ can be chosen arbitrarily small. 
\end{proof}

\ww{Note that if we use above only one value $k = k_1$, then the upper bound for $\mu$ in Theorem 5 is greater than $d$. Therefore, to get sharp approximation exponents we necessarily need to use several values of $k$.}

Now we shall prove Theorems \ref{thm:1}, \ref{thm:TM} and \ref{thm:2}.

\begin{proof}[Proof of Theorem \ref{thm:1}]
We choose $d=2$, $f(z)=A(z)$ and $g(z)=B(z)$, and apply Theorem \ref{lem:2}. By (\ref{eqn:s}), $\phi_{1}(z)=1$, $\phi_{2}(z)=-(1+z+z^{2})$, $\phi_{3}(z)=0$, $\psi_{1}(z)=1$, $\psi_{2}(z)=1+z+z^{2}$ and $\psi_{3}(z)=-2$. Thus the conditions given before Theorem \ref{lem:2} are satisfied. Further, $\phi(z)=1+z+z^{2}$ and $\hat{\phi}_{2}(z)=1$, $\hat{\psi}_{2}(z)=-1$, $v=2$.

For all $k$, $7\leq k \leq 51$, the determinant 
\[
\begin{vmatrix}
\Delta_{k,k+1,k-1}\\
\delta_{1}
\end{vmatrix}\neq 0,
\]
where we use the notations of subsection \ref{sec:pade} (see \cite[Appendix A]{VM2015}). This implies the condition (i) of Theorem \ref{lem:2}. Moreover, as noted in subsection \ref{sec:pade}, we have $\ww{\mathfrak{o}}(k)=\mathrm{ord}R_{k}(z)=3k+2$. Now we may take $\bar{e}(k)=k$ and $\tau=2$ in (\ref{eq:x}), so $e_{1}=1$ and $e_{2}=0$. Assume now that $k_{1}$, $8\leq k_{1}\leq 25$, is given and choose $k_{2}=k_{1}+1$, $k_{3}=k_{1}+2, \dots, k_{t}=2k_{1}+1\ (t=k_{1}+2)$. Then also the condition (ii) of Theorem \ref{lem:2} is satisfied. By this Theorem
\begin{eqnarray*}
\mu\left(A\left(\frac{a}{b}\right),B\left(\frac{a}{b}\right)\right) & \leq & 
\max\left\{\max_{1\leq j \leq t-1}\frac{(1-\lambda)(3k_{j}+5)}{2k_{j}-\lambda(3k_{j}+2)}, \frac{(1-\lambda)(3k_{t}+1)}{2k_{t}-\lambda(3k_{t}+2)}\right\}\\
& = & \frac{(1-\lambda)(3k_{1}+5)}{2k_{1}-\lambda(3k_{1}+2)}, 
\end{eqnarray*}
if $\lambda<\frac{2k_{1}}{3k_{1}+2}$. The choice $k_{1}=25$ gives Theorem \ref{thm:1}.
\end{proof}

\begin{remark}\label{rem:1}
\ww{By the discussion in subsection \ref{sec:pade} the above proof would give $\mu(A(1/b),B(1/b)) = 3/2$, if one could prove that the determinant 
\begin{equation}
\begin{vmatrix}
\Delta_{k,k+1,k-1}\\ \delta_{1} 
\end{vmatrix} \neq 0,  \label{eqn:det2}
\end{equation}
for all $k\geq k_{0}$}. However, the determinants (\ref{eqn:det2}) are more complicated than the Hankel determinants of one function $A(z)$ or $B(z)$ used in the consideration of $\mu(A(1/b))$ and $\mu(B(1/b))$. For the research of such Hankel determinants see \cite{BHWY2015,Coons2013,GWW2014,HW2014,Keijo2015,WW2014} and references there in.
\end{remark}

\begin{proof}[Proof of Theorem \ref{thm:TM}]
In this proof we use Theorem \ref{lem:2} with $d=2$, $f(z)=T(z)$ and $g(z)=M(z)$. Now $\phi_{1}(z)=1$, $\phi_{2}(z)=z-1$, $\phi_{3}(z)=0$, $\psi_{1}(z)=z-1$, $\psi_{2}(z)=1$, $\psi_{3}(z)=z(1-z)$. Thus $\phi(z)=1-z$, $\hat{\phi}_{2}(z)=1-z$ and $\hat{\psi}_{2}(z)=-1$, which gives $v=3$.

We construct $(k,k,k+1)$ approximations with $9$ values $k=k_{j}$ given in the following table:
\begin{center}
\begin{tabular}{c|ccccccccc}
$j$ & 1 & 2 & 3 & 4 & 5 & 6 & 7 & 8 & 9 \\
\hline
$k_{j}$ & 8 & 9 & 10 & 11 & 12 & 13 & 14 & 15 & 16 \\
\hline 
 $\ww{\mathfrak{o}}(k_{j})$ & 32 & 32 & 33 & 36 & 39 & 42 & 45 & 48 & 52 \\
\end{tabular}
\end{center}

The polynomials $A_{k}$, $B_{k}$ and also $\ww{\mathfrak{o}}(k):=\mathrm{ord}R_{k}(z)$ are given in  \cite[Appendix B]{VM2015}, where we also see that all $A_{k}(0)\neq 0$. Again, in this special case we have in (\ref{eq:x}) $\bar{e}(k)=k+1$, $\tau=1$ giving $e_{1}=0$, $e_{2}=2$. Thus the condition (\emph{ii}) of Theorem \ref{lem:2} holds, and we obtain
\begin{eqnarray*}
\mu\left(T\left(\frac{a}{b}\right),\, M\left(\frac{a}{b}\right)\right)
& \leq & \max_{1\leq j\leq 9}\frac{(1-\lambda)\ww{\mathfrak{o}}(k_{j+1})}{\ww{\mathfrak{o}}(k_{j})-(k_{j}+2)-\lambda \ww{\mathfrak{o}}(k_{j})}\\
&=& \frac{2(1-\lambda)\ww{\mathfrak{o}}(k_{1})}{\ww{\mathfrak{o}}(k_{9})-(k_{9}+2)-\lambda \ww{\mathfrak{o}}(k_{9})}
= \frac{32(1-\lambda)}{17-26\lambda}, 
\end{eqnarray*}
if $\lambda<1/2$. This proves Theorem \ref{thm:TM}.

\end{proof}

In a similar way, we can prove Theorem \ref{thm:2}.
\begin{proof}[Proof of Theorem \ref{thm:2}]
In this case $d=3$, $f(z)=G_{3}(z)$ and $g(z)=F_{3}(z)$.  By the functional equations (\ref{eqn:G}), $\phi_{1}(z)=z-1$, $\phi_{2}(z)=1-z$, $\phi_{3}(z)=z$, $\psi_{1}(z)=-(1+z)$, $\psi_{2}(z)=1+z$, $\psi_{3}(z)=z$. Thus $\phi(z)=1-z^{2}$, $\hat{\phi}_{2}(z)=1-z$ and $\hat{\psi}_{2}(z)=1+z$. We have $v=2$.

We shall use $(k,k,k)$ approximations with $6$ values of $k$, which are given in \cite[Appendix C]{VM2015} and satisfy (\emph{i}) in Theorem \ref{lem:2}. The important parameters are here
\begin{center}\ww{
\begin{tabular}{c|ccccccccc}
$j$ & 1 & 2 & 3 & 4 & 5 & 6 \\
\hline
$k_{j}$ & 9 & 10 & 13 & 18 & 22 & 26\\
\hline
$\mathfrak{o}(k_{j})$ & 29 & 36 & 45 & 56 & 70 & 80
\end{tabular}}
\end{center}
In (\ref{eq:x}) we now have $\bar{e}(k)=k$, $\tau=2$. So $e_{1}=e_{2}=0$, and again the condition (\emph{ii}) is satisfied. By Theorem \ref{lem:2} it follows that 
\begin{eqnarray*}
\mu\left(G_{3}\left(\frac{a}{b}\right),F_{3}\left(\frac{a}{b}\right)\right) & \leq & \max_{1\leq j\leq 6}\frac{(1-\lambda)\ww{\mathfrak{o}}(k_{j+1})}{(1-\lambda)\ww{\mathfrak{o}}(k_{j})-\bar{e}(k_{j})-\frac{\tau}{d-1}}\\
& = & \frac{(1-\lambda)\ww{\mathfrak{o}}(k_{2})}{(1-\lambda)\ww{\mathfrak{o}}(k_{1})-\bar{e}(k_{1})-1}=\frac{36(1- \lambda)}{19-29\lambda},
\end{eqnarray*}
if $\lambda<19/29$, which proves Theorem \ref{thm:2}.

\end{proof}

\section{Proof of Theorem \ref{thm:self}}
The function $S(z)$ satisfies the functional equation 
\[S(z^{16})=-zS(z)+(1+z+z^{2})S(z^{4}).\]
Therefore, starting from $(k,k,k-1)$ Hermite-Pad\'{e} approximation 
\[A_{k}(z)S(z)+B_{k}(z)S(z^{4})+C_{k}(z)=R_{k}(z)\]
we obtain an infinite sequence of approximations
\begin{equation}
A_{k,m}(z)S(z)+B_{k,m}(z)S(z^{4})+C_{k,m}(z)=R_{k,m}(z),\ m=0,1,\dots \label{eqn:self:5}
\end{equation}
where $A_{k,0}(z)=A_{k}(z)$, $B_{k,0}(z)=B_{k}(z)$, $C_{k,0}(z)=C_{k}(z)$, $R_{k,0}(z)=R_{k}(z)$, and for $m\geq 0$,
\begin{equation}
\begin{array}{rclrcl}
A_{k,m+1}(z) &=& -zB_{k,m}(z^{4}), &  B_{k,m+1}(z) & =& (1+z+z^{2})B_{k,m}(z^{4})+A_{k,m}(z^{4}),\\
C_{k,m+1}(z) & = & C_{k,m}(z^{4}),  & R_{k,m+1}(z) & = & R_{k,m}(z^{4}). \label{eqn:self:6}
\end{array}
\end{equation}
By the above recursions (\ref{eqn:self:6}), $\deg A_{k,m}$ and $\deg B_{k,m}$ are at most $k\cdot 4^{m}+2(1+4+\cdots +4^{m-1})=k\cdot 4^{m}+2(4^{m}-1)/3$, and $\deg C_{k,m}\leq k\cdot 4^{m}$. 

We need to estimate the absolute values of $A_{k,m}(z)$ and $B_{k,m}(z)$ at $z=a/b$, $|a|<b$. For these considerations we assume that $A_{k}(0)=0$, $B_{k}(0)\neq 0$. Then there exists $m_0=m_0(k,a/b)$ such that
\begin{equation}
\left|A_{k}(z^{4^{m}})\right|\leq \left|z^{2\cdot 4^{m-1}}\right|\cdot\left|B_{k}(z^{4^{m}})\right|, \qquad \frac{|B_{k}(0)|}{2}\leq \left|B_{k}(z^{4^{m}})\right|\leq \frac{3|B_{k}(0)|}{2} \label{eqn:self:7}
\end{equation}
for all $m\geq m_{0}$ and $-|a|/b\leq z \leq |a|/b$. Then, by (\ref{eqn:self:6}) and the first inequality in (\ref{eqn:self:7}), 
\begin{eqnarray*}
\left|B_{k,1}(z^{4^{m-1}})\right| & \leq &  (1+z^{4^{m-1}}+z^{2\cdot 4^{m-1}})\left|B_{k}(z^{4^{m}})\right|+\left|A_{k}(z^{4^{m}})\right| \leq (1+z^{4^{m-1}}+2z^{2\cdot 4^{m-1}})\left|B_{k}(z^{4^{m}})\right|, \\
\left|B_{k,1}(z^{4^{m-1}})\right| & \geq &  (1+z^{4^{m-1}}+z^{2\cdot 4^{m-1}})\left|B_{k}(z^{4^{m}})\right|-\left|A_{k}(z^{4^{m}})\right| \geq(1+z^{4^{m-1}})\left|B_{k}(z^{4^{m}})\right|,\\
\left|A_{k,1}(z^{4^{m-1}})\right| & = & \left|z^{4^{m-1}}\right|\cdot \left|B_{k}(z^{4^{m}})\right|\leq \frac{z^{4^{m-1}}}{1+z^{4^{m-1}}}\left|B_{k,1}(z^{4^{m-1}})\right|. %
%{\color{gray!80}\leq z^{2\cdot 4^{m-2}}\left|B_{k,1}(z^{4^{m-1}})\right|}.
\end{eqnarray*}
Repeating this we get
\begin{eqnarray*}
\left|B_{k,2}(z^{4^{m-2}})\right| & \leq & (1+z^{4^{m-2}}+2z^{2\cdot 4^{m-2}})\left|B_{k,1}(z^{4^{m-1}})\right|, \\
\left|B_{k,2}(z^{4^{m-2}})\right| & \geq & (1+z^{4^{m-2}})\left|B_{k,1}(z^{4^{m-1}})\right|,\\
\left|A_{k,2}(z^{4^{m-2}})\right| & \leq  & \left|z^{4^{m-2}}\right|\cdot \left|B_{k,1}(z^{4^{m-1}})\right|\leq \frac{z^{4^{m-2}}}{1+z^{4^{m-2}}}\left|B_{k,2}(z^{4^{m-2}})\right|.
\end{eqnarray*}
After $m$ steps we have 
\begin{eqnarray*}
\left|B_{k,m}(z)\right| & \leq & \left|B_{k}(z^{4^{m}})\right|\prod_{l=0}^{m-1}(1+z^{4^{l}}+2z^{2\cdot 4^{l}}),\\
\left|B_{k,m}(z)\right| & \geq &  \left|B_{k}(z^{4^{m}})\right|\prod_{l=0}^{m-1}(1+z^{4^{l}}),\\
\left|A_{k,m}(z)\right| & \leq & \frac{|z|}{1+z}\left|B_{k,m}(z)\right|.
\end{eqnarray*}
By (\ref{eqn:self:7}) we therefore obtain, for all $m\geq m_{0}$, 
\begin{eqnarray*}
\left|B_{k,m}\left(\frac{a}{b}\right)\right| & \leq & \frac{3|B_{k}(0)|}{2}\prod_{l=0}^{\infty}\left(1+\left(\frac{|a|}{b}\right)^{4^{l}}+2\left(\frac{|a|}{b}\right)^{2\cdot 4^{l}}\right)=:\hat{c}_{1}(k),\\
\left|B_{k,m}\left(\frac{a}{b}\right)\right| & \geq & \frac{|B_{k}(0)|}{2}\prod_{l=0}^{\infty}\left(1-\left(\frac{|a|}{b}\right)^{4^{l}}\right)=:\hat{c}_{2}(k),\\
\left|A_{k,m}\left(\frac{a}{b}\right)\right| & \leq & \frac{|a|}{b-|a|}\left|B_{k,m}\left(\frac{a}{b}\right)\right|. 
\end{eqnarray*}

In our proof we shall use the following values.
\begin{center}
\begin{tabular}{c|cccccccccc}
$j$ & 1 & 2 & 3 & 4 & 5 & 6 & 7 & 8 & 9 & 10\\
\hline
$k_{j}$ & 16 & 21 & 27 & 32 & 37 & 42 & 47 & 52 & 57 & 63 \\
\hline 
 $\ww{\mathfrak{o}}(k_{j})$ & 64 & 64 & 82 & 108 & 112 & 127 & 172 & 172 & 172 & 190\\
\end{tabular}
\end{center}
In all these cases $A_{k_{j}}(0)=0$ and $B_{k_{j}}(0)\neq 0$, see  \cite[Appendix D]{VM2015}. We now construct linear forms $r(k_{j},m)$, by multiplying (\ref{eqn:self:5}), where $z=a/b$ and $k=k_{j}$, with 
\[Q_{k_{j}, m}=b^{(k_{j}+\frac{2}{3})\cdot 4^{m}-\frac{2}{3}}.\]

For a given $\delta>0$ there exists $m_{1}> 0$ such that
\begin{eqnarray*}
\frac{\left(k_{j+1}+\frac{2}{3}\right)\cdot {4^{m}}-\frac{2}{3}}{\left(k_{j}+\frac{2}{3}\right)\cdot {4^{m}}-\frac{2}{3}} & \leq & \frac{k_{j+1}+\frac{2}{3}}{k_{j}+\frac{2}{3}}+\delta, \ j=1,2,\dots, 9 \\
\frac{\left(k_{1}+\frac{2}{3}\right)\cdot {4^{m+1}}-\frac{2}{3}}{\left(k_{10}+\frac{2}{3}\right)\cdot {4^{m}}-\frac{2}{3}} & < & \frac{4(k_{1}+\frac{2}{3})}{k_{10}+\frac{2}{3}}+\delta,  
\end{eqnarray*}
for all $m\geq m_{1}$. By the above consideration the conditions (\ref{eqn:l1}), (\ref{eqn:l2}) and (\ref{eqn:l3}) of Lemma \ref{lem:1} are satisfied for all $m\geq \max\{m_{0},m_{1}\}$ if we choose
\[\theta(j)=\frac{k_{j+1}+\frac{2}{3}}{k_{j}+\frac{2}{3}}+\delta, \ j=1,2,\dots, 9, \text{ and } 
\theta(10)= \frac{4(k_{1}+\frac{2}{3})}{k_{10}+\frac{2}{3}}+\delta.\]
Moreover, we may choose $m_{1}$ above in such a way that also (\ref{eqn:l4}) is satisfied with 
\[\alpha(k_{j})-\delta=\beta(k_{j})=\frac{(1-\lambda)\ww{\mathfrak{o}}(k_{j})-k_{j}-\frac{2}{3}}{k_{j}+\frac{2}{3}}, \ j=1,2,\dots, 10\]
for all $m\geq m_{1}$. 

Since we may choose $\delta$ above arbitrarily small and 
\begin{align*}
& \ww{\max\left\{\max_{1\leq j \leq 9}\frac{(1-\lambda)\ww{\mathfrak{o}}(k_{j+1})}{(1-\lambda)\ww{\mathfrak{o}}(k_{j})-(k_{j}+\frac{2}{3})},\, \frac{4(1-\lambda)\ww{\mathfrak{o}}(k_{1})}{(1-\lambda)\ww{\mathfrak{o}}(k_{10})-(k_{10}+\frac{2}{3})}\right\}}\\
=\, & \frac{(1-\lambda)\ww{\mathfrak{o}}(k_{7})}{(1-\lambda)\ww{\mathfrak{o}}(k_{6})-(k_{6}+\frac{2}{3})}
 =  \frac{172(1-\lambda)}{127(1-\lambda)-(42+\frac{2}{3})},
\end{align*}
if $\lambda < 178/291=0.611\dots$, Theorem \ref{thm:self} follows from Lemma \ref{lem:1}.\hfill $\square$

\section*{Appendices}
\appendix
\section{Values of determinant (\ref{eqn:det2}) for $7\leq k \leq 51$}
For $7\leq k \leq 51$, values of determinant (\ref{eqn:det2}) ($\mathrm{mod}~ 49$) are:
37, 13, 3, 10, 6, 22, 24, 47, 13, 19, 46, 47, 27, 2, 44, 28, 28, 12, 20, 34, 30, 5, 5, 46, 2, 39, 35, 44, 14, 4, 12, 47, 10, 2, 31, 36, 13, 16, 43, 46, 7, 5, 21, 15, 21.

\section{Approximation polynomials in Theorem \ref{thm:TM}}
Here we list the approximation polynomials $A_{k}(z)$ and $B_{k}(z)$ and the order $\ww{\mathfrak{o}}(k)$ of  $R_{k}(z)$.
\begin{longtable}{c||c||>{\small}b{0.8\textwidth}}
\hline
\multirow{2}{*}{$k$} & \multirow{2}{*}{$\ww{\mathfrak{o}}(k)$} &$A_{k}(z)$ \\   \cline{3-3}
                     &                         & $B_{k}(z)$ \\   \hline
                     \endhead
\multirow{2}{*}{8} & \multirow{2}{*}{32} & $z^{8}+2\, z^{4}+1$ \\   \cline{3-3}
                   &                    & $-2\, z^{6}+4\, z^{5}+2\, z^{4}-8\, z^{3}+2\, z^{2}+4\, z-2$ \\   \hline
\multirow{2}{*}{9} & \multirow{2}{*}{32} & $5\,{z}^{9}-{z}^{8}+2\,{z}^{5}-10\,{z}^{4}-4\,{z}^{3}-4\,{z}^{2}+z-5$ \\   \cline{3-3}
                   &                    & $16\,{z}^{9}-16\,{z}^{8}-10\,{z}^{7}+22\,{z}^{6}-26\,{z}^{5}-10\,{z}^{4
}+34\,{z}^{3}+2\,{z}^{2}-6\,z-6$ \\   \hline
\multirow{2}{*}{10} & \multirow{2}{*}{33} & ${z}^{9}+4\,{z}^{5}+2\,{z}^{4}+{z}^{3}+{z}^{2}+2\,z+1$ \\   \cline{3-3}
                   &                    & $-4\,{z}^{9}+4\,{z}^{8}-2\,{z}^{7}+4\,{z}^{6}+10\,{z}^{5}-16\,{z}^{4}-2\,{z}^{3}+8\,{z}^{2}-4\,z+2$ \\   \hline                                      
\multirow{2}{*}{11} & \multirow{2}{*}{36} & $2\,{z}^{10}+2\,{z}^{6}+{z}^{4}+2\,{z}^{2}+1$ \\   \cline{3-3}
                   &                    & $4\,{z}^{10}-8\,{z}^{9}+8\,{z}^{7}-4\,{z}^{6}+2\,{z}^{2}-4\,z+2$ \\   \hline
\multirow{2}{*}{12} & \multirow{2}{*}{39} & $2\,{z}^{12}+2\,{z}^{11}-2\,{z}^{10}-2\,{z}^{8}-{z}^{6}+{z}^{3}+z-1$ \\   \cline{3-3}
                   &                    & $12\,{z}^{12}-16\,{z}^{11}-8\,{z}^{10}+16\,{z}^{9}-12\,{z}^{8}+12\,{z}^{7}+4\,{z}^{6}-12\,{z}^{5}+10\,{z}^{4}-8\,{z}^{3}-2\,{z}^{2}+6\,z-2
$ \\   \hline
\multirow{2}{*}{13} & \multirow{2}{*}{42} & $2\,{z}^{13}+2\,{z}^{12}+4\,{z}^{11}+4\,{z}^{10}+4\,{z}^{9}+4\,{z}^{8}+2\,{z}^{7}+2\,{z}^{6}+3\,{z}^{5}+3\,{z}^{4}+3\,{z}^{3}+3\,{z}^{2}+2\,z
+2$ \\   \cline{3-3}
                   &                    & $8\,{z}^{11}-8\,{z}^{10}-4\,{z}^{9}+4\,{z}^{8}-12\,{z}^{7}+12\,{z}^{6}+8\,{z}^{5}-8\,{z}^{4}+6\,{z}^{3}-6\,{z}^{2}-4\,z+4$ \\   \hline
\multirow{2}{*}{14} & \multirow{2}{*}{45} & $10\,{z}^{14}+16\,{z}^{13}+20\,{z}^{12}+24\,{z}^{11}+20\,{z}^{10}+20\,{z}^{9}+11\,{z}^{8}+8\,{z}^{7}+13\,{z}^{6}+18\,{z}^{5}+18\,{z}^{4}+20\,{z}^{3}+14\,{z}^{2}+12\,z+4$ \\   \cline{3-3}
                   &                    & $12\,{z}^{14}-8\,{z}^{13}+20\,{z}^{12}-8\,{z}^{11}-44\,{z}^{10}+16\,{z}^{9}-28\,{z}^{8}+16\,{z}^{7}+74\,{z}^{6}-20\,{z}^{5}-8\,{z}^{4}-8\,{z}^{3}-36\,{z}^{2}+8\,z+16$ \\   \hline
\multirow{2}{*}{15} & \multirow{2}{*}{48} & ${z}^{15}+{z}^{14}-{z}^{11}-{z}^{10}-{z}^{9}-{z}^{8}+{z}^{7}+{z}^{6}-{z
}^{3}-{z}^{2}-z-1$ \\   \cline{3-3}
                   &                    & $2\,{z}^{13}-2\,{z}^{12}-6\,{z}^{11}+6\,{z}^{10}+14\,{z}^{7}-14\,{z}^{6}-8\,{z}^{5}+8\,{z}^{4}-8\,{z}^{3}+8\,{z}^{2}+6\,z-6$ \\   \hline
\multirow{2}{*}{16} & \multirow{2}{*}{52} & ${z}^{16}+6\,{z}^{15}+11\,{z}^{14}+16\,{z}^{13}+17\,{z}^{12}+18\,{z}^{11}+15\,{z}^{10}+12\,{z}^{9}+7\,{z}^{8}+10\,{z}^{7}+13\,{z}^{6}+16\,{z
}^{5}+15\,{z}^{4}+14\,{z}^{3}+10\,{z}^{2}+6\,z+1$ \\   \cline{3-3}
                   &                    & $4\,{z}^{16}+2\,{z}^{14}+8\,{z}^{13}-4\,{z}^{12}-16\,{z}^{11}-10\,{z}^{
10}-8\,{z}^{9}+8\,{z}^{8}+32\,{z}^{7}+14\,{z}^{6}-8\,{z}^{5}-12\,{z}^{
4}-16\,{z}^{3}-6\,{z}^{2}+8\,z+6
$ \\   \hline
\end{longtable} 

\section{Approximation polynomials in Theorem \ref{thm:2}}
\ww{
\begin{longtable}{c||c||>{\small}b{0.8\textwidth}}
\hline
\multirow{2}{*}{$k$} & \multirow{2}{*}{$\ww{\mathfrak{o}}(k)$} &$A_{k}(z)$ \\   \cline{3-3}
                     &                         & $B_{k}(z)$ \\   \hline
                     \endhead
\multirow{2}{*}{$9$} & \multirow{2}{*}{$29$} &$-179\,{z}^{9}+{z}^{8}+{z}^{7}+14\,{z}^{6}+{z}^{5}+{z}^{4}+14\,{z}^{3}+{z}^{2}+z+141$ \\   \cline{3-3}
                     &                         & ${z}^{9}+{z}^{8}-{z}^{7}+14\,{z}^{6}-{z}^{5}+{z}^{4}-14\,{z}^{3}+{z}^{2}-z-37$ \\   
\hline
\multirow{2}{*}{$10$} & \multirow{2}{*}{$36$} &${z}^{10}+{z}^{9}-z-1$ \\   \cline{3-3}
                     &                         & ${z}^{10}+{z}^{9}+z+1$ \\   
\hline
\multirow{2}{*}{$13$} & \multirow{2}{*}{$45$} &$26\,{z}^{13}-2\,{z}^{11}-2\,{z}^{9}-4\,{z}^{7}-15\,{z}^{4}+{z}^{2}-4\,z+1$ \\   \cline{3-3}
                     &                         & $-26\,{z}^{13}+2\,{z}^{11}+2\,{z}^{9}+4\,{z}^{7}-15\,{z}^{4}+{z}^{2}+4\,z+1$ \\   
\hline
\multirow{2}{*}{$18$} & \multirow{2}{*}{$56$} &$142\,{z}^{18}-{z}^{16}-{z}^{14}-14\,{z}^{12}-{z}^{10}-52\,{z}^{9}-{z}^{8}-14\,{z}^{6}-{z}^{4}-{z}^{2}-52$ \\   \cline{3-3}
                     &                         & $-142\,{z}^{18}+{z}^{16}+{z}^{14}+14\,{z}^{12}+{z}^{10}-52\,{z}^{9}+{z}
^{8}+14\,{z}^{6}+{z}^{4}+{z}^{2}+52$ \\   
\hline
\multirow{2}{*}{$22$} & \multirow{2}{*}{$70$} &$38\,{z}^{22}-3\,{z}^{20}-{z}^{19}-3\,{z}^{18}-4\,{z}^{16}-14\,{z}^{13}+{z}^{11}-{z}^{10}+{z}^{9}-14\,{z}^{4}+{z}^{2}-z+1
$ \\   \cline{3-3}
                     &                         & $-38\,{z}^{22}+3\,{z}^{20}-{z}^{19}+3\,{z}^{18}+4\,{z}^{16}-14\,{z}^{13
}+{z}^{11}+{z}^{10}+{z}^{9}+14\,{z}^{4}-{z}^{2}-z-1$ \\   
\hline
\multirow{2}{*}{$26$} & \multirow{2}{*}{$80$} &$-379\,{z}^{26}+{z}^{25}+{z}^{24}+11\,{z}^{23}+{z}^{22}+{z}^{21}+11\,{z
}^{20}+{z}^{19}+{z}^{18}+141\,{z}^{17}+{z}^{16}+{z}^{15}+11\,{z}^{14}+
{z}^{13}+{z}^{12}+11\,{z}^{11}+{z}^{10}+{z}^{9}+141\,{z}^{8}+{z}^{7}+{
z}^{6}+11\,{z}^{5}+{z}^{4}+{z}^{3}+11\,{z}^{2}+z+1$ \\   \cline{3-3}
                     &                         & $379\,{z}^{26}+{z}^{25}-{z}^{24}+11\,{z}^{23}-{z}^{22}+{z}^{21}-11\,{z}
^{20}+{z}^{19}-{z}^{18}+141\,{z}^{17}-{z}^{16}+{z}^{15}-11\,{z}^{14}+{
z}^{13}-{z}^{12}+11\,{z}^{11}-{z}^{10}+{z}^{9}-141\,{z}^{8}+{z}^{7}-{z
}^{6}+11\,{z}^{5}-{z}^{4}+{z}^{3}-11\,{z}^{2}+z-1
$ \\   
\hline
\end{longtable} 
}

\section{Approximation polynomials in Theorem \ref{thm:self}}

\begin{longtable}{c||c||>{\small}b{0.8\textwidth}}
\hline
\multirow{2}{*}{$k$} & \multirow{2}{*}{$\ww{\mathfrak{o}}(k)$} &$A_{k}(z)$ \\   \cline{3-3}
                     &                         & $B_{k}(z)$ \\   \hline
                     \endhead
\multirow{2}{*}{$16$} & \multirow{2}{*}{$64$} &${z}^{15}-{z}^{13}+{z}^{11}+{z}^{7}+z$ \\   \cline{3-3}
                     &                         & $-{z}^{16}-{z}^{15}+{z}^{13}-{z}^{11}-2\,{z}^{8}-{z}^{7}-{z}^{6}+{z}^{4}-{z}^{2}-z-1$ \\   \hline
\multirow{2}{*}{$21$} & \multirow{2}{*}{$64$} &${z}^{20}+2\,{z}^{16}+{z}^{14}-{z}^{13}+{z}^{12}-{z}^{4}-{z}^{2}+z$ \\   \cline{3-3}
                     &                         & $-{z}^{21}-{z}^{20}-{z}^{19}-2\,{z}^{17}-2\,{z}^{16}-2\,{z}^{15}-{z}^{13}+{z}^{9}-{z}^{8}-{z}^{7}+2\,{z}^{4}+2\,{z}^{3}-1$ \\   \hline
\multirow{2}{*}{$27$} & \multirow{2}{*}{$82$} & ${z}^{27}-{z}^{25}+{z}^{23}-2\,{z}^{21}+{z}^{19}-2\,{z}^{9}-2\,{z}^{5}-z$ \\   \cline{3-3}
                     &                         & $-{z}^{27}+{z}^{25}-{z}^{24}-{z}^{23}+2\,{z}^{22}+2\,{z}^{21}-{z}^{19}-
{z}^{18}+2\,{z}^{10}+2\,{z}^{9}+2\,{z}^{8}+2\,{z}^{6}+2\,{z}^{5}+{z}^{
2}+z+1$ \\   \hline
\multirow{2}{*}{$32$} & \multirow{2}{*}{$108$} &${z}^{29}+2\,{z}^{25}+{z}^{21}+2\,{z}^{13}+4\,{z}^{9}+4\,{z}^{5}+2\,z$ \\   \cline{3-3}
                     &                         & $-{z}^{32}-{z}^{30}-{z}^{29}-{z}^{28}-2\,{z}^{26}-2\,{z}^{25}-{z}^{24}-
{z}^{22}-{z}^{21}-2\,{z}^{14}-2\,{z}^{13}-2\,{z}^{12}-4\,{z}^{10}-4\,{
z}^{9}-2\,{z}^{8}-4\,{z}^{6}-4\,{z}^{5}-2\,{z}^{4}-2\,{z}^{2}-2\,z-1
$ \\   \hline
\multirow{2}{*}{$37$} & \multirow{2}{*}{$112$} &${z}^{37}+{z}^{33}-2\,{z}^{31}-4\,{z}^{30}-{z}^{29}+2\,{z}^{25}+{z}^{23
}+2\,{z}^{22}+2\,{z}^{21}-2\,{z}^{19}-4\,{z}^{18}-2\,{z}^{17}-4\,{z}^{
15}-8\,{z}^{14}-2\,{z}^{13}-2\,{z}^{11}-4\,{z}^{10}+2\,{z}^{9}+3\,{z}^
{5}+{z}^{3}+2\,{z}^{2}+2\,z$ \\   \cline{3-3}
                     &                         & ${z}^{37}-{z}^{34}-{z}^{33}+{z}^{32}+6\,{z}^{31}+6\,{z}^{30}+3\,{z}^{29
}-3\,{z}^{26}-4\,{z}^{25}-3\,{z}^{24}-3\,{z}^{23}-4\,{z}^{22}-2\,{z}^{
21}+2\,{z}^{20}+6\,{z}^{19}+8\,{z}^{18}+6\,{z}^{17}+6\,{z}^{16}+12\,{z
}^{15}+12\,{z}^{14}+6\,{z}^{13}+2\,{z}^{12}+6\,{z}^{11}+2\,{z}^{10}-2
\,{z}^{9}-2\,{z}^{8}-2\,{z}^{6}-{z}^{5}-{z}^{4}-3\,{z}^{3}-5\,{z}^{2}-
4\,z-2$ \\   \hline
\multirow{2}{*}{$42$} & \multirow{2}{*}{$127$} &${z}^{42}+{z}^{41}-{z}^{40}+{z}^{38}+{z}^{37}-2\,{z}^{36}+{z}^{34}+{z}^
{33}+{z}^{32}+{z}^{26}+{z}^{25}-3\,{z}^{24}+{z}^{22}+{z}^{21}-2\,{z}^{
20}+{z}^{10}+{z}^{9}+{z}^{6}+{z}^{5}-{z}^{4}+{z}^{2}+z$ \\   \cline{3-3}
                     &                         & $-2\,{z}^{42}-{z}^{41}-{z}^{39}-2\,{z}^{38}+{z}^{37}+2\,{z}^{36}-2\,{z}
^{34}-3\,{z}^{33}-2\,{z}^{32}-{z}^{31}-2\,{z}^{26}+{z}^{25}+2\,{z}^{24
}+{z}^{23}-2\,{z}^{22}+{z}^{21}+2\,{z}^{20}-2\,{z}^{10}-2\,{z}^{9}-{z}
^{8}-2\,{z}^{7}-2\,{z}^{6}+{z}^{4}-2\,{z}^{2}-2\,z-1$ \\   \hline
\multirow{2}{*}{$47$} & \multirow{2}{*}{$172$} &$-z^{46}+z^{45}-z^{42}+2z^{41}-z^{38}+2z^{37}+z^{33}-z^{30}+z^{29}-z^{26}+2z^{25}-z^{17}-z^{14}-z^{10}+2z^{9}-z^{6}+z^{5}$ \\   \cline{3-3}
                     &                         & ${z}^{47}-{z}^{44}+{z}^{43}-{z}^{42}-2\,{z}^{41}-{z}^{40}+{z}^{39}-{z}^
{38}-{z}^{37}-{z}^{36}-{z}^{34}-{z}^{33}-{z}^{32}+{z}^{31}-{z}^{28}+{z
}^{27}-{z}^{26}-2\,{z}^{25}-{z}^{24}+{z}^{20}+{z}^{18}+{z}^{17}+{z}^{
15}+{z}^{14}+{z}^{13}+{z}^{11}-{z}^{10}-2\,{z}^{9}-{z}^{8}+{z}^{7}-1$ \\   \hline
\multirow{2}{*}{$52$} & \multirow{2}{*}{$172$} &$ -z^{51}-z^{48}-2z^{44}-z^{43}+z^{41}-2z^{40}+z^{37}-z^{36}-2z^{35}+z^{33}-z^{32}+z^{25}-z^{24}-z^{17}-3z^{16}+z^{15}-z^{13}-z^{11}+z^{9}-2z^{8}+z^{7}-3z^{4}$ \\   \cline{3-3}
                     &                         & $z^{52}+z^{51}+z^{50}+z^{49}+z^{48}+z^{47}-z^{46}+2z^{45}+3z^{44}+2z^{43}+z^{42}+z^{41}+z^{40}+z^{39}-2z^{38}+3z^{36}+3z^{35}+z^{34}+z^{31}-z^{30}-z^{27}+2z^{23}-z^{22}+z^{20}-2z^{19}+2z^{18}+4z^{17}+2z^{16}+2z^{15}-z^{14}+z^{13}+2z^{12}+z^{10}+z^{9}+2z^{7}-2z^{6}+3z^{5}+4z^{4}+z^{3}+z^{2}-1$ \\   \hline
\multirow{2}{*}{$57$} & \multirow{2}{*}{$172$} &$-{z}^{53}-2\,{z}^{49}+7\,{z}^{17}+5\,{z}^{13}-4\,{z}^{9}+7\,{z}^{5}+5\,z$ \\   \cline{3-3}
                     &                         & ${z}^{54}+{z}^{53}+{z}^{52}+2\,{z}^{50}+2\,{z}^{49}+{z}^{48}-{z}^{44}+2
\,{z}^{40}-{z}^{36}-{z}^{32}+2\,{z}^{28}-2\,{z}^{24}-{z}^{20}-7\,{z}^{
18}-7\,{z}^{17}-2\,{z}^{16}-5\,{z}^{14}-5\,{z}^{13}-3\,{z}^{12}+4\,{z}
^{10}+4\,{z}^{9}+2\,{z}^{8}-7\,{z}^{6}-7\,{z}^{5}-8\,{z}^{4}-5\,{z}^{2
}-5\,z+1$ \\   \hline
\multirow{2}{*}{$63$} & \multirow{2}{*}{$190$} &$-{z}^{62}+{z}^{58}+2\,{z}^{57}-2\,{z}^{55}+{z}^{54}+2\,{z}^{53}-2\,{z}
^{51}+3\,{z}^{50}+3\,{z}^{49}-3\,{z}^{47}+{z}^{42}+{z}^{41}-{z}^{39}+{
z}^{38}+{z}^{21}-{z}^{19}+3\,{z}^{18}+3\,{z}^{17}-3\,{z}^{15}+{z}^{14}
-{z}^{9}-{z}^{6}-2\,{z}^{5}+{z}^{3}-{z}^{2}-z$ \\   \cline{3-3}
                     &                         & $ {z}^{63}+{z}^{62}+{z}^{61}-{z}^{59}-3\,{z}^{58}-4\,{z}^{57}+{z}^{55}-{
z}^{54}-{z}^{53}+2\,{z}^{52}-{z}^{51}-6\,{z}^{50}-6\,{z}^{49}+3\,{z}^{
47}+3\,{z}^{46}+{z}^{45}+{z}^{44}-{z}^{43}-3\,{z}^{42}-3\,{z}^{41}-{z}
^{40}+{z}^{38}+{z}^{36}-{z}^{34}+{z}^{29}+{z}^{28}-{z}^{26}-{z}^{25}-{
z}^{24}+{z}^{20}-2\,{z}^{19}-6\,{z}^{18}-6\,{z}^{17}+2\,{z}^{15}+2\,{z
}^{14}+{z}^{12}+{z}^{9}+{z}^{7}+4\,{z}^{6}+2\,{z}^{5}+{z}^{2}+2\,z+1$ \\   \hline                                                                                                                                                                                             
\end{longtable}

%\bibliographystyle{plain}
%\bibliography{thebib}

\noindent
\begin{tabular}{l}
Keijo V\"a\"an\"anen\\ 
Department of Mathematical Sciences,\\ 
University of Oulu, \\
P.O. Box 3000, 90014 Oulu, Finland.\medskip\\
\texttt{keijo.vaananen@oulu.fi}\\
\end{tabular}\hfill
\begin{tabular}{l}
Wen Wu\\ 
Department of Mathematical Sciences, \\
University of Oulu,\\
P.O. Box 3000, 90014 Oulu, Finland.\medskip\\
\texttt{hust.wuwen@gmail.com}
\end{tabular}

\end{document}